\documentclass[12pt]{amsart}
\usepackage{amssymb,latexsym,amscd,verbatim,color}
\usepackage[T1]{fontenc}
\usepackage{mathrsfs} 
\usepackage{xy}
\usepackage{bbm} 
\xyoption{all}
\usepackage{ulem}
\usepackage{floatflt}
 
\usepackage[colorlinks,linkcolor=blue,citecolor=blue,urlcolor=black]{hyperref}

\swapnumbers
\numberwithin{equation}{section}
\newtheorem{thm}{Theorem}[subsection]
\newtheorem{propose}[thm]{Proposition}
\newtheorem{lemma}[thm]{Lemma}
\newtheorem{cor}[thm]{Corollary}
\theoremstyle{definition}

\newtheorem{rem}[thm]{Remark}

\newcommand{\Def}{\mathrm{Def}(R,k)}
\newcommand{\M}{\mathcal{M}_{1}}

\newcommand{\spec}{\mathrm{Spec}} 
  
\newcommand{\ab}{\mathrm{ab}}

\newcommand{\Hom}{\operatorname{Hom}}      
\newcommand{\Ext}{\operatorname{Ext}}      

\newcommand{\mm}{\mathfrak{m}} 
\newcommand{\E}{\mathbb{E}}  

\newcommand{\Q}{\mathbb{Q}}     
\newcommand{\Z}{\mathbb{Z}}  
\newcommand{\N}{\mathbb{N}}
\newcommand{\G}{\mathbb{G}} 
  
\newcommand{\V}{\mathbb{V}}
\newcommand{\cB}{\mathcal{B}}  

\renewcommand{\ker}{\operatorname{Ker}}  
\newcommand{\by}[1]{\stackrel{#1}{\rightarrow}}
\newcommand{\longby}[1]{\stackrel{#1}{\longrightarrow}}

\newcommand{\et} {{\rm \acute{e}t}}   
\DeclareMathOperator{\gr}{gr}

\begin{document}
\title{On deformations of  $1$-motives}

\author{A. Bertapelle}
\address{Dipartimento di Matematica ``Tullio Levi-Civita'', Universit\`a   degli Studi di Padova\\ via Trieste, 63\\ I-35121 Padova, Italy} \email{alessandra.bertapelle@unipd.it}
\author{N. Mazzari }
\address{Institut de Math\'ematiques de Bordeaux\\
351, cours de la Lib\'eration\\ F-33405 Talence cedex \\ France}
\email{nicola.mazzari@u-bordeaux.fr}
\date{\today}
\begin{abstract}
According to a well-known theorem of Serre and Tate, the infinitesimal deformation theory of an abelian variety in  positive characteristic is equivalent to the infinitesimal deformation theory of its Barsotti-Tate group. We extend this result to $1$-motives.
\end{abstract}

\maketitle

 \section{Introduction}
Let $R$ be an artinian local ring with maximal ideal $\mm$ and  perfect  residue field $k$ of positive characteristic  $p$. 
Let $\M(R)$ denote  the category of $1$-motives over $R$. 
For any  $1$-motive $M$ (respectively Barsotti\,-Tate group $\cB$) over $R$   let $M_{0}$ (respectively $\cB_{0}$) denote its base change to $k=R/\mm$. 
Let  $\Def$ denote the category whose objects are triples $(M_{0}, \cB,\varepsilon_0)$ where $M_{0}$ is a $1$-motive over $k$, $\cB$ is a  Barsotti\,-Tate group over $R$ and $\varepsilon_{0}\colon \cB_{0}\to M_{0}[p^\infty]$ is an isomorphism of $\cB_0$ with the BT group of $M_0$. 
Aim of this paper is to prove the following 
\begin{thm}\label{t.st}
The functor 
\begin{eqnarray}\label{st}
\Delta_{R}\colon \M(R)&\to &\Def\\
\nonumber M&\mapsto & (M_{0}, M[p^\infty], \text{ natural } \varepsilon_0),
\end{eqnarray} is an equivalence of categories.  
\end{thm} 
This result generalizes the well-known Serre\,-Tate equivalence for abelian schemes  over artinian local rings $R$ as above (cf. \cite[Theorem 1.2.1]{Ka},  \cite[V, Theorem 2.3]{M}). 
Arguments in {\it loc. cit.} do not extend directly to the case of $1$-motives (neither to the case of semi-abelian varieties, see Theorem \ref{t.stab}) but are used to get some intermediate results.
Note that the proof of fullness (Proposition \ref{full}) and essential  surjectiveness (Proposition \ref{surj}) uses weights, a carefully study the case of $1$-motives of the form $[\Z\to \G_m]$ (Lemma \ref{l.tm}), which is orthogonal to the classical case of abelian schemes, and Galois descent arguments (Lemma \ref{l.desc}).

As a consequence, we can extend in Section \ref{moduli} the description of the formal moduli space of an abelian variety with ordinary reduction over an algebraically closed field to $1$-motives.

\subsection*{Notation} Let   $p^s,  s\geq 1$, denote the characteristic of $R$. Then $R$ is canonically endowed with a structure of finite $W_s(k)$-algebra,  where $W_s(k)=W(k)/(p^s)$ denotes the ring of Witt vectors of length $s$. Let $s' \geq s$ be the minimal integer such that $\mm^{s'}=0$.
We also fix a positive integer $n$ such that $(1+\mm)^{p^n}=\{1\}$. For any $R$-group scheme $G$ we will identify $G(R)=\mathrm{Hom}_R(\spec\, R,G)$ with $\mathrm{Hom}_{R\text{-gr}}(\Z,G)$ by mapping $a\in G(R)$ to the morphism $u\colon \Z\to G$ such that $u(1)=a$.
\section{Barsotti\,-Tate groups and $1$-motives}
Let $\M(R)$ be the category of $1$-motives over $R$. Its  objects are two term complexes of commutative $R$-group schemes  $M=[u\colon L\to G]$ where $G$ is extension of an abelian scheme $A$ by a torus $T$, and $L$ is locally for the \'etale topology on $R$  isomorphic to $\Z^r$.  Recall that a $1$-motive is naturally filtered as follows 
\[
 0\subseteq W_{-2}M=[0\to T]\subseteq W_{-1}M=[0\to G]\subseteq W_0M=M
	\ .
\]
Since morphisms of $1$-motives respect filtrations,   $\M(R)$ is a filtered category. We will denote by  $\M(R)_{\le i}$  the full subcategory of $\M(R)$ consisting of $1$-motives such that $M=W_iM$.
Given a $1$-motive $M$, let  $M_\ab=[u_\ab\colon L\to A]$, where $u_\ab$ is the composition of $u$ with the canonical morphism $G\to A$. 
Let $\M(R)_{\geq-1}$ be the full subcategory of $\M(R)$ consisting of $1$-motives $M=M_\ab$.  Similarly $\Def_{\geq-1}$ (respectively  $\Def_{\leq-1}$) denotes the full subcategory of $\Def$ consisting of objects such that $M_0$ is in $\M(k)_{\geq-1}$ (respectively, in  $\M(k)_{\leq-1}$).
\smallskip

Any  $1$-motive $M=[u\colon L\to G]$ over $R$ admits a so-called universal extension $M^\natural=[u^\natural\colon L\to G^\natural]$ which fits into a short exact sequence of complexes  of $R$-group schemes 
\begin{equation}
\label{vmm}
0\to [0\to \V(M)]\to M^\natural\to M\to 0\end{equation} 
where $\V(M)$ is the vector group associated to the dual sheaf of $\underline{\mathrm{Ext}}^1_{\mathrm{Zar}}(M,\G_{a,R})$. 
Note that $\V(M)$ is killed by $p^s$ since the  multiplication by $p^s$ morphism is the $0$ map on $\G_{a,R}$. We recall that $M^\natural$ is denoted by $\E(M)=[L\to \E(M)_G]$ in \cite{AB,ABV}.
 \smallskip
 
For any $m\in \N$ consider the cone of the multiplication by $m$ on $M$
\[M/ mM\colon \quad  L\overset{{-u}\choose{-m}}{\longrightarrow} G\oplus L\overset{(-m,u)}{\longrightarrow} G,\]  
where $L$ is in degree $-2$, and let $M[m]=H^{-1}(M/ mM)$ (see \cite[10.1.4]{De} or \cite[\S 1.3]{ABV} with a different sign convention); it is a  finite and flat $R$-group scheme and it fits in the following diagram
\begin{equation} \label{dia.big}
\xymatrix@R=0.6cm{
&&& L\ar@{=}[r]\ar@{^{(}->}[d]^{{-u}\choose{-m}}& L \ar@{^{(}->}[d]^{-m}& 
\\
\tilde \eta_{M[m]}\colon& 0\ar[r]& G[m] \ar[r]\ar@{=}[d]&    \ker((-m,u)) \ar@{->>}[d]\ar[r]& L \ar[r]\ar@{->>}[d]& 0
\\
\eta_{M[m]}\colon& 0\ar[r]& G[m] \ar[r]&  M[ m] \ar[r]& L/m L \ar[r]& 0
} \end{equation}

\begin{rem}\label{rm} 
One checks by direct computations that the pull-back of  
\[\xi_{G[m]}\colon \quad 0\to G[m]\to G\by{m} G\to 0
\] along $u$ is isomorphic up to sign to $\tilde \eta_{M[m]} $. 
\end{rem} 

Since the push-out of $\tilde\eta_{M[p^r]}$ along $G[p^r]\to G[p^{r+1}]$ is isomorphic to $p\cdot \tilde \eta_{M[p^{r+1}]}$, there is  a morphism of complexes $\eta_{M[p^r]}\to \eta_{M[p^{r+1}]}$ for any $r\geq 1$. One then gets  a short exact sequence of  Barsotti\,-Tate groups (BT groups for short) 
\begin{equation}\label{e.pi}
\eta_{M[p^\infty]}\colon \quad 0\to G[p^\infty]\to M[p^\infty]\to L[p^\infty]\to 0,\end{equation}
where $L[p^\infty]:=L\otimes_\Z\Q_p/\Z_p$.
 
\begin{lemma}\label{l.fil} Let  $M_0=[u_0\colon L_0\to G_0]$ and $\varphi_0\colon M_0\to N_0$  a morphism of $1$-motives over $k$.  
\begin{enumerate}
\item[(i)] Any finite and flat $R$-group scheme $M[m]$ which lifts $M_0[m]$ is naturally endowed with a weight filtration which lifts the weight filtration $T_0[m]\subseteq G_0[m]\subseteq M_0[m]$. If $M_0$ lifts to a $1$-motive $M$ over $R$,  then the natural weight filtration on $M[m]$ agrees with the one induced by the weight filtration on $M$.
\item[(ii)] Any morphism of finite flat $R$-group schemes $M[m]\to N[m]$ which lifts $\varphi_0[m]$ respects filtrations. 
\item[(iii)] Any  BT group over $R$ which lifts $M_0[p^\infty]$ is naturally endowed with a weight filtration and any morphism of BT groups over $R$ which lifts $\varphi_0[p^\infty]$ respects filtrations.
\end{enumerate}
\end{lemma}
\begin{proof}
Let $L$ be the unique \'etale lifting of $L_0$ over $R$. Let $M[m]$ be  a finite and flat $R$-group scheme lifting $M_0[m]$. Then $M[m]$ fits into a short exact sequence
\[0\to G[m]\to M[ m]\to L/m L\to 0\] 
which lifts the sequence $\eta_{M_0[m]}$   for $M_0$ (cf. \eqref{dia.big}). 
Indeed, there exists a unique lifting $f\colon M[ m]\to L/m L$ of $M_0[m]\to L_0/m L_0$ since $L/mL$ is \'etale and we set $G[m]=\ker f$. Note further that $f$ factors through the epimorphism  $\pi_M\colon M[m]\to M[m]^\et=M[m]/M[m]^\circ$ with $ M[m]^\circ$  the identity component of $M[m]$. Let $f^\et\colon M[m]^\et\to L/m L$ satisfy $f^\et\circ \pi_M=f$. Since $\pi_M$ is an fppf epimorphism, it suffices to prove that the same is $f^\et$. This follows immediately since   $f^\et_0$  is an epimorphism of finite \'etale $k$-group schemes.   

Note further that $M[m]$ is naturally filtered. Indeed, let $T$ be the unique torus lifting $T_0$, $L^*$ its group of characters  and $M[m]^*$ the Cartier dual of $M[m]$. We have seen that the canonical morphism $M_0[m]^*\to L_0^*/mL^*_0$  lifts uniquely over $R$ and hence, by Cartier duality, the immersion  $T_0[m]\to  M_0[m]$ lifts uniquely over $R$. Now, the composition $T[m]\to  M[m]\to L/mL$ is the zero map, since its reduction modulo $\mm$ is the zero map and $L/mL$ is \'etale. Hence the closed immersion  $T [m]\to  M [m]$ factors through $G[m]$ and (i) is proved.

With similar arguments one shows (ii).  Finally   (iii) follows immediately from (i) and (ii).
\end{proof}

The classical Serre\,-Tate theorem says that   deformations of an abelian variety over $k$ only depend on   deformations of its  BT group. As we will now explain, the analogous result for $1$-motives of the type $[\Z\to \G_{m,k}]$ can  explicitly be deduced from the canonical  isomorphism 
\begin{equation}\label{e.rkm}
R^*\longby{\sim} k^*\times (1+\mm), \quad x\mapsto (x_0, x/[x_0]),
\end{equation}
where $x_0$ is the reduction of $x$ modulo $\mm$ and $[x_0]\in W_{s}(k)$ is the multiplicative representative of $x_0$. Recall that  $(1+\mm)^{p^n}=\{1\}$ and hence  \[1+\mm=\mu_{p^n}(R)=\mu_{p^\infty}(R).\]

We now consider BT groups which are extensions of    $\Q_p/\Z_p$  by  $\mu_{p^\infty}$. 
By \cite[Proposition 2.5 p. 180]{M} there exists an isomorphism 
\begin{equation}\label{e.psi}
\Psi\colon 1+\mm\by{\sim}  \mathrm{Ext}_R(\Q_p/\Z_p,\mu_{p^\infty})
\end{equation}  
which associates to $a\in 1+\mm$ the  push-out along $a\colon \Z\to \mu_{p^\infty} $ of the sequence \[0\to \Z\to \varinjlim_m \frac 1{p^m}\Z\to \Q_p/\Z_p\to 0 .\]
Let $\mathrm{Ext}_{\Z/p^n\Z}(-,-)$ denote classes of extensions of  $R$-group schemes killed by $p^n$.

\begin{lemma}\label{l.psi}
There is an isomorphism of groups $\Psi_n\colon 1+\mm\by{\sim} \mathrm{Ext}_{\Z/p^n\Z}(\Z/p^n\Z, \mu_{p^n})$ which associates to $a\in 1+\mm$ the short exact sequence $\Psi(a)[p^n]$, i.e.,  the short exact sequence of kernels of the multiplication by $p^n$ associated  to $\Psi(a)$.
\end{lemma}
\begin{proof}
By \cite[p.183]{M} we have isomorphisms
\begin{multline}\label{e.4e} \mathrm{Ext}_R(\Q_p/\Z_p,\mu_{p^\infty}) \simeq
\varprojlim_m \mathrm{Ext}_R(\Z/p^m\Z ,\mu_{p^\infty})\simeq \\ \mathrm{Ext}_R(\Z/p^n\Z  ,\mu_{p^\infty})\simeq  \mathrm{Ext}_{\Z/p^n\Z}(\Z/p^n\Z, \mu_{p^n}),  \end{multline} 
where the first isomorphism is induced by pull-back  along $\Z/p^m\Z\to \Q_p/\Z_p$, the second isomorphism follows since $p^n$ kills $ \mathrm{Ext}_R(\Q_p/\Z_p,\mu_{p^\infty})$, and the third isomorphism is obtained by passing to kernels of the multiplication by $p^n$. By definition $\Psi_n$ is the composition of the isomorphism $\Psi$ in \eqref{e.psi} with the isomorphisms in \eqref{e.4e}. Hence the conclusion follows.
\end{proof} 

\begin{rem}\label{r.push}  Note that $\Psi_n(a)$ is also  represented by the sequence of kernels of the multiplication by $p^n$ of the sequence  $\widetilde{\Psi}(a) $ obtaied via push-out along $a\colon \Z\to \G_{m,R}$ from $\zeta \colon 0\to \Z\to  \Z\to \Z/p^n\Z\to 0 $. Indeed,
$\widetilde{\Psi}(a) =\iota_{n}^*\iota_*\Psi(a)$ with  $\iota_n\colon \Z/p^n\Z\to \Q_p/\Z_p$ and $\iota\colon \mu_{p^\infty}\to \G_{m,R}$ the canonical morphisms.
\end{rem}

We now see how to interpret \eqref{e.rkm} in terms of deformations of $1$-motives.

\begin{lemma}\label{l.tm}
Given a $1$-motive $M_0=[u_0\colon \Z\to \G_{m,k}]$ and a $\mathcal B\in  \mathrm{Ext}_R(\Q_p/\Z_p,\mu_{p^\infty})$ there is a unique $1$-motive $M=[u\colon \Z\to \G_{m,R}]$  which lifts $M_0$ and whose BT group is isomorphic to  $\mathcal B$ as extension of $\mu_{p^\infty} $ by $\Q_p/\Z_p$.
\end{lemma}
\begin{proof}
Note that if $\mathcal B,\mathcal B'$ are two  liftings of $M_0[p^\infty]$  there is at most one morphism $\mathcal B\to \mathcal B'$ as extensions of $\Q_p/\Z_p$  by  $\mu_{p^\infty}$.

We have to prove that the homomorphism \[
\Phi\colon R^*\to k^*\times \mathrm{Ext}_R(\Q_p/\Z_p,\mu_{p^\infty}),\] which maps $u\in R^*$ to the pair $(u_0,M^u[p^\infty])$, where $u_0$ is the reduction of $u$ modulo $\mm$ and $M^u=[u\colon \Z\to \G_{m,R}]$, is an isomorphism. 
Note that if  $u=[u_0], u_0\in k^*$, then $u$  admits $p^r$th roots for any $r\geq 1$; hence   $\eta_{M^u[p^\infty]} $ is split since by \eqref{e.psi} $\mathrm{Ext}_R(\Q_p/\Z_p,\mu_{p^\infty}) $  is killed by $p^n$. We are then reduced  to checking  that 
\[\Phi_{|1+\mm}\colon 1+\mm\to \mathrm{Ext}_R(\Q_p/\Z_p,\mu_{p^\infty}),\]
is an isomorphism. Thanks to \eqref{e.4e} we may view it as  the homomorphism
\[\Phi_n\colon 1+\mm\to \mathrm{Ext}_{\Z/p^n\Z}(\Z/p^n\Z, \mu_{p^n})  \]  
which maps $u$ to $\eta_{M^u[p^n]}$. Since $u$ is $p^n$-torsion, we may write it as a morphism $u\colon \Z/p^n\Z\to \G_{m,R}$. Let  $u^D\colon \Z\to \mu_{p^n}$ be the morphism obtained by applying Cartier duality to $u$.  By Remark \ref{rm} $\Phi_n(u)=\eta_{M^u[p^n]}$ is, up to sign, the sequence of cokernels of the multiplication by $p^n$ of the sequence obtained via pull-back along $u\colon \Z\to \G_{m,R}$ from the sequence $0\to \mu_{p^n}\to \G_{m,R}\to \G_{m,R}\to 0$. Passing to Cartier duals and recalling Remark \ref{r.push}, we then get that $\Phi(u)^D=-\Psi_n(u^D)$. 
The conclusion follows now from Lemma \ref{l.psi} and the fact that Cartier duality induces isomorphisms on both $1+\mm=\Hom_{R-{\rm gr}}(\Z/p^n\Z, \mu_{p^n})$ and $\mathrm{Ext}_R(\Q_p/\Z_p,\mu_{p^\infty}) $. 
\end{proof}
 
\section{Intermediate results}
%

\begin{rem}\label{rmbis}  We can refine some preliminary results in \cite{Ka}.
\begin{enumerate}
\item[(a)] Let $G$ be any smooth commutative $R$-group scheme. By \cite[Lemma 1.1.2]{Ka}  the kernel of the reduction map $G(R)\to G(k)$ is killed by $p^{s(s'-1)}$. In our hypothesis one can prove that it is killed by $p^{s'-1}$. 
Indeed, by the theory of Greenberg functor the sections $G(R/\mm^i)$, $1\leq i\leq s'$, can be identified with the group of $k$-rational sections of a smooth $k$-group scheme $\mathrm{Gr}_i(G)$. Further there are so-called change of level morphisms $\varrho_{i}^1\colon \mathrm{Gr}_{i+1}(G)\to \mathrm{Gr}_i(G)$ such that $\varrho_i^1(k)$ coincides with the reduction map $G(R/\mm^{i+1})\to G(R/\mm^i)$. By Greenberg's structure theorem (\cite{Gr},\cite[Thm. 12.13]{BGA})  the kernel of $\varrho_{i}^1$ is a $k$-vector group, thus killed by $p$. It follows then by induction that $\ker (G(R)\to G(k))$ is killed by $p^{s'-1}$.
\item[(b)] By \cite[Lemma 1.1.3 (3)]{Ka}  a morphism of semi-abelian varieties $\varphi_{0}\colon G_{0}\to H_{0}$ lifts over $R$ up to multiplication by $p^{s(s'-1)}$. We can improve also this estimate if $p>2$. Let $M$, $N$ be $1$-motives over $R$. By \cite[Thm. 2.1]{AB} there exists a canonical morphism  between universal extensions  \eqref{vmm} $\varphi^\natural\colon M^\natural\to N^\natural$ which lifts $\varphi_0^\natural$. 
If $\varphi^\natural(\V(M))\subseteq \V(N)$, then  $\varphi^\natural$ induces a morphism $\varphi\colon M\to N$ which lifts $\varphi_{0}$. In general, since the multiplication by $p^s$ morphism kills $\V(M)$,  the morphism $p^s\varphi^\natural$ maps $\V(M)$ to $0$. Hence $p^s\varphi_{0}$ lifts to a morphism $\text{``$p^s\varphi$''}\colon M\to N$.
\end{enumerate}
 \end{rem}
\subsection{Galois actions}\label{gal}
Let $k'/k$ be a finite Galois extension of Galois group $\Gamma=\mathrm{Gal}(k'/k)$ and set $R'=R\otimes_{W(k)}W(k')$. Note that $R'$ is an artinian local ring with residue field $k'$ and $\spec\, R'\to \spec\, R$ is a Galois covering of group $\Gamma$.  
Then $\Gamma$ naturally acts on $\M(R')$ and $\mathrm{Def}(R',k')$ and Serre\,-Tate functor $\Delta_{R'}$ \eqref{st} commutes with the Galois action.

Note that the datum of a $1$-motive $M$ over $R$ is equivalent to the datum of a $1$-motive $M'=[u'\colon L'\to G']$ over $R'$ with a $\Gamma$-action compatible with the $\Gamma$-action on $ \spec\, R'$.  Indeed $L'$ descends to a   lattice over $R$. Further, since $G'$ is a separated $R'$-scheme  homeomorphic to the semi-abelian $k'$-variety  $G'_0$, it  can be covered by affine open subschemes which are stable under the $\Gamma$-action and thus it descends  to a $R$-group scheme. The maximal subtorus $T'$ of $G'$  descends to a subtorus $T$ of $G$ and $G/T$ is an abelian scheme since it is an abelian scheme after base change to $R'$.
Similarly, the datum of an object $(G_0, \mathcal B,\varepsilon_0)$ in   $\mathrm{Def}(R,k)$ is equivalent to the datum of a  $(G'_0, \mathcal B',\varepsilon'_0)$ in   $\mathrm{Def}(R',k')$ together with a $\Gamma$-action compatible with the $\Gamma$-action on $k'$ (in the first and third entries) and the $\Gamma$-action on $R'$ (on the second entry).
\subsection{Serre\,-Tate theorem for semi-abelian varieties} 
 Recall that Cartier duality is a self duality on $\M(R)$ which provides  a contravariant equivalence between the categories $\M(R)_{\leq -1}$ and $\M(R)_{\geq -1}$ \cite[10.2]{De}, \cite[\S 1]{ABV}. Further if $M^*$ is the Cartier dual of a $1$-motive $M$, $M^*[p^\infty]$ is the Cartier dual of the BT group $M[p^\infty]$. 
 
Consider the functor 
\begin{eqnarray}\label{stab}
\M(R)_{\leq -1}&\to &\Def_{\leq -1} \\
\nonumber G &\mapsto & (G_{0}, G[p^\infty], \text{ natural } \varepsilon_0),
\end{eqnarray}  induced by \eqref{st}. 
The proof of the next theorem contains several arguments used for the proof of the more general Serre\,-Tate Theorem \ref{t.st}.  The fact that  \eqref{stab} is an equivalence of categories is needed later to prove both the fullness and the essential surjectiveness of \eqref{st}.  For this reason we can  not jump directly to the general case.

\begin{thm}\label{t.stab} 
The functor \eqref{stab} is an equivalence of categories.
\end{thm}
\begin{proof}  The faithfulness can be proved as in \cite[p. 144]{Ka} for the case of abelian varieties. Let us give an alternative proof when $p>2$. 
If $\varphi\colon G\to H$ is a morphism in $\M(R)_{\leq -1}$ and $\varphi_0=0$ then  the induced morphism between universal extensions  $\varphi^\natural\colon G^\natural\to H^\natural$ is the zero map \cite[Thm. 2.1]{AB} and hence $\varphi=0$.  We now prove the fullness when $p>2$; if $p=2$ the same proof works replacing $s$ with $s(s'-1)$ and it coincides with the one in \cite[p. 144]{Ka}. Given a $\varphi_0\colon G_0\to H_0$  and a morphism $\psi\colon G[p^\infty]\to H[p^\infty]$ lifting $\varphi_0[p^\infty]$ by Remark \ref{rmbis} (b) 
there exists a morphism $\text{``$p^s\varphi$''}\colon G\to H$ lifting $p^s\varphi_0$. Further $p^s\psi=\text{``$p^s\varphi$''}[p^\infty]$ by \cite[Lemma 1.1.3(2)]{Ka} since both morphisms lift $p^s\varphi_0[p^\infty]$. Hence $\text{``$p^s\varphi$''}$ kills $G[p^s]$ and thus  $\text{``$p^s\varphi$''}=p^s\varphi$ for a  morphism $\varphi\colon G\to H$ (necessarily unique since $\mathrm{Hom}_{R\text{-gr}}(G,H)$ has no $p$-torsion). Thus the fully faithfulness is proved.

Let now $(G_0, \mathcal B,\varepsilon_0)$ be an object in $\Def_{\leq -1}$.
If the maximal subtorus of $G_0$ is split of dimension $d$, then $G_0$ lifts to an $R$-group scheme $G$ which is extension of an abelian scheme by a $\G_{m,R}^d$ and, for proving the essential surjectivity, one can repeat word by word the proof of the classical Serre\,-Tate theorem \cite[Thm. 1.2.1]{Ka}.
Indeed   the Cartier dual of $G_0$ is a $1$-motive $[w_0\colon \Z^d\to A^*_0]$ with $A_0^*$ an abelian variety, which lifts to an abelian scheme $A^*$ over $R$. Hence $w_0$ lifts to a morphism $w\colon \Z^d\to A^*$ over $R$ and by applying again Cartier duality, one gets $G$.

For the general case, one proceeds by Galois descent as follows (cf. Subsection \ref{gal}).  Let $k'/k$ be a finite Galois extension such that the maximal subtorus of $G_0$ becomes split over $k'$. Let $\Gamma=\mathrm{Gal}(k'/k)$ and set $R'=R\otimes_{W(k)}W(k')$. Let $(G'_0, \mathcal B',\varepsilon'_0)$ in  $\mathrm{Def}(R',k')_{\leq -1}$ be obtained  from $(G_0, \mathcal B,\varepsilon_0)$  via  base change along  $k'/k$ (on semi-abelian varieties) and $R'/R $ (on BT groups).
For any $\sigma\in \Gamma$ we  have an isomorphism in $\mathrm{Def}(R',k')_{\leq -1}$
\[
(\varphi_{\sigma,0},\psi_\sigma)\colon \left(G'_0,\mathcal B',\mathrm{can}\right)\by{\sim} \left(\sigma^*G'_0,\sigma^*\mathcal B',\sigma^*\mathrm{can}\right),
  \]
 with \[\varphi_{\sigma,0}\colon G'_0\by{\sim}\sigma^* G'_0, \qquad
\psi_\sigma\colon \mathcal B'\by{\sim} \sigma^* \mathcal B',\qquad \varphi_{\sigma,0}[p^\infty]\circ \mathrm{can}=(\sigma^*\mathrm{can})\circ \psi_{\sigma,0}\ .\] 

By the previous step we know that there exists a  $G'$ in $\M(R')_{\leq -1}$ whose image in   $\mathrm{Def}(R',k')_{\leq -1}$ is the triple  $(G'_0, \mathcal B',\varepsilon'_0)$.
By the fully faithfulness of \eqref{stab} the isomorphism $(\varphi_{\sigma,0},\psi_\sigma)$ gives a unique isomorphism $\varphi_\sigma\colon  G'\by{\sim}\sigma^* G'$ which lifts $\varphi_{\sigma,0}$ and restricts to $\psi_\sigma$ on BT groups. 
These morphisms  define a $\Gamma$-action on $G'$ compatible with the action on $\spec\, R'$ and hence $G'$ descends  to  a $G$ in $\M(R)_{\leq -1}$ with BT group isomorphic to $\mathcal B$.
\end{proof}

Theorem \ref{t.stab} allows another step towards the proof of Theorem \ref{t.st}.

\begin{cor}\label{c.ab}
The functor $\M(R)_{\geq-1}\to \Def_{\geq-1}$ induced by \eqref{st} is an equivalence of categories.
\end{cor}
\begin{proof}
 It is sufficient to apply Cartier duality to the previous result.
\end{proof}

\section{The general case}
\subsection{The faithfulness}

The proof of the faithfulness of \eqref{st}  follows easily from faithfulness on $\M(R)_{\leq -1} $.

\begin{propose}\label{p.inj} Let $M, N $ be two $1$-motives over $R$. Then the reduction map
\[\Hom_{\M(R)}(M,N)\to \Hom_{\M(k)}(M_{0},N_{0})\]
is injective.
\end{propose}
\begin{proof}
Consider a morphism $\varphi\colon M\to N$ and assume that its reduction modulo $\mm$ is the $0$ morphism. Hence  $\varphi=0$ in degree $-1$ by the equivalence of the \'etale sites over $R$ and over $k$ and $\varphi=0$ in degree $0$ by  Theorem \ref{t.stab}.
\end{proof}

\begin{cor}\label{faith}
The functor  \eqref{st} is faithful.
 \end{cor}

\subsection{The fullness}

For the proof we will need the following lemma.

\begin{lemma}\label{l.help}
Let $M=[u\colon L\to G]$ be a $1$-motive over $R$. If $u_0=0$ and $M[p^\infty]$ is split extension of $L[p^\infty]$ by $G[p^\infty]$, then $u=0$.
\end{lemma}
\begin{proof} We may work \'etale locally on $R$ and assume $L$ constant and $T$ a split torus. By Corollary \ref{c.ab} $u_\ab\colon L\to A$ is the $0$ morphism. Hence $u$ factors through $T$. Let $M_t=[u\colon L\to T]$ and note that  $M_t[p^\infty]$ is split  extension $L[p^\infty]$ by $T[p^\infty]$.   Hence,  $u=0$ by Lemma \ref{l.tm}. 
\end{proof}
  
\begin{propose}\label{full}
The functor \eqref{st} is full.  
\end{propose} 
\begin{proof}
Let $M=[u\colon L\to G], N=[v\colon F\to H]$ be two $1$-motives over $R$, $\varphi_0\colon M_{0}\to N_{0}$  a morphism between their reduction modulo $\mm$   and $\psi\colon  M[p^\infty]\to N[p^\infty]$ a lifting of $\varphi_0[p^\infty]\colon  M_{0}[p^\infty]\to N_{0}[p^\infty]$.  We have to prove that 
there exists a  lifting $\varphi\colon M\to N$ of $\varphi_{0}$ over $R$ such that $\varphi[p^\infty]=\psi$.

Let $\varphi_0=(f_0,g_0)$, i.e., $f_0\colon L_0\to F_0, g_0\colon G_0\to H_0$ and $g_0\circ u_0=v_0\circ f_0$. Recall that any lifting $\psi$ of $\varphi_0[p^\infty]$ respects weight filtrations  by Lemma \ref{l.fil}. Hence, by Theorem \ref{t.stab}  there exists a unique morphism $g\colon G\to H$ lifting  $g_0$. Further, by the equivalence between the category of \'etale group schemes over $k$ and the category of \'etale group schemes over $R$, there exists a unique $f\colon L\to F$ lifting $f_0$.
We are left to prove that $g\circ u=v\circ f$, so that $\varphi=(f,g)$ is a morphism of $1$-motives, and that $\varphi[p^\infty]=\psi$. The latter equality follows from \cite[Lemma 1.1.3 2)]{Ka} since both morphisms are liftings of $\varphi_0[p^\infty]$. 

Let  $Z=[w\colon L\to H]$ with $w= g\circ u-v\circ f$. We claim that 
$\eta_{Z[p^\infty]}$ in \eqref{e.pi} is split. For proving this claim, it is sufficient to check that $\eta_{Z[p^r]}$ (and hence $\tilde \eta_{Z[p^r]}$ in \eqref{dia.big}) is split for any $r$.  Since $\psi[p^r]$ restricts to the morphism $g[p^r]\colon G[p^r]\to H[p^r]$ on weight $\leq -1$ subgroups  and induces $f/p^rf\colon L/p^rL\to F/p^rF$ in weight $0$,  it is 
\[(f/p^rf)^*\eta_{N[p^r]}=g[p^r]_*\eta_{M[p^r]},
\]
and hence  
\[f^*\tilde\eta_{N[p^r]}=g[p^r]_*\tilde\eta_{M[p^r]}.
\] 
We conclude  then  by applying Remark \ref{rm} that
\begin{eqnarray*}
\tilde\eta_{Z[p^r]}=
(-w)^*\xi_{H[p^r]}&=& f^*(v^*\xi_{H[p^r]})-u^*(g^*\xi_{H[p^r]})=
f^*\tilde\eta_{N[p^r]}-u^*g[p^r]_*\xi_{G[p^r]}\\&=&f^*\tilde\eta_{N[p^r]}-g[p^r]_*\tilde\eta_{M[p^r]}  =0.\end{eqnarray*}
Hence $Z$  is a $1$-motive  such that $\eta_{Z[p^\infty]}$ is split extension of $L[p^\infty]$ by $H[p^\infty]$ and $w_0=0$. Thus $w=0$ by Lemma \ref{l.help}. 
\end{proof}

\subsection{Essential surjectiveness}

The strategy of the proof is first to construct the desired $1$-motive \'etale locally and then to apply descent. By Theorem \ref{t.stab}, Corollary \ref{c.ab} and Lemma \ref{l.tm} we have already partial results in this direction.
We prove another special case. 
\begin{lemma}\label{l.lz} Let $M_0=[u_0\colon \Z^m\to G_0]$ be a $1$-motive over $k$ with $T_0=\G_{m,k}^d$ a split torus. Let $\cB$ be a BT group over $R$ lifting $M_0[p^\infty]$. Then there exists a unique (up to unique isomorphism) $1$-motive $M$ over $R$ lifting $M_0$ and with BT group isomorphic to $\cB$.
\end{lemma}
\begin{proof} Uniqueness follows by Proposition \ref{p.inj}.
For the existence,  recall that by Lemma \ref{l.fil} $\cB$ is naturally filtered so that $W_{-1}\cB$ is a lifting of $G_{0}[p^\infty]$ and $\mathcal B/W_{-2}\mathcal B$ is a lifting of $M_{0,\ab}[p^\infty]$. By Theorem \ref{t.stab}  $G_0$ lifts  to an $R$-scheme $G$ which is extension of an abelian scheme $A$ by  $T\simeq \G_{m,R}^d$ and  $G[p^\infty]=W_{-1}\cB$; further  $M_{0,\ab}$ lifts to a $1$-motive $M_A=[u_A\colon \Z^m\to A]$ whose BT group is isomorphic to $\mathcal B/W_{-2}\mathcal B$ by Corollary \ref{c.ab}.

Let  $M'=[u'\colon \Z^m\to G]$  be any extension of $M_A$ by $T$; it exists since $H^1(R,\G_{m,R})=0$. Since $T(R)\to T(k)$ is surjective, we may assume that $u'$ is also a lifting of $u_0$. 
We are then left to alter $u'$ so that $M'[p^\infty]\simeq \mathcal B[p^\infty]$ in $\Ext_R((\Q_p/Z_p)^m,G[p^\infty])$.

Let  $\mathcal E=\mathcal B[p^n]-M'[p^n]$ as extension of  $(\Z/p^n\Z)^m$ by $G[p^n]$. 
Since the push-out along $ G[p^n]\to A[p^n]$ maps  $\mathcal E$ to the trivial extension of $(\Z/p^n\Z)^m$ by $A[p^n]$, there exists by Lemma \ref{l.tm} a morphism $v\colon \Z^m\to T$ such that $M=[u=u'-v\colon \Z^m\to G]$ is a lifting of $M_0$ and $M[p^n]$ is isomorphic to $\mathcal B[p^n]$.  Further $M[p^\infty]$ and $\mathcal B$ induce the same extension of $(\Q_p/\Z_p)^m$ by $A[p^\infty]$. Hence they differ by an $\eta\in \mathrm{Ext}_R((\Q_p/\Z_p)^m,\mu_{p^\infty}^d)$ which induces a split extension on the $p^n$-torsion subgroups. 
By \eqref{e.4e}  we conclude that $\eta=0$ and hence $M[p^\infty]\simeq \mathcal B$.\end{proof}

Let $k'/k$ be a finite Galois extension with Galois group $\Gamma$ and set $R'=R\otimes_{W(k)}W(k')$.  As remarked in Subsection \ref{gal}, if a $1$-motive $M'$ over $R'$ descends to a $1$-motive $M$ over $R$, then   $\Delta_{R'}(M')$ in  $\mathrm{Def}(R',k')$ descends to   $\Delta_R(M)$ in $\mathrm{Def}(R,k)$.  The next lemma shows that also the converse holds. 

\begin{lemma}\label{l.desc}
Let notation be as above. A $1$-motive $M'$ descends to $R$ if, and only if, its image in  $\mathrm{Def}(R',k')$ descends to $\mathrm{Def}(R,k)$.
\end{lemma}
\begin{proof}
Let $M'=[u'\colon L'\to G']$ and assume that $(M'_0,\mathcal B'=M'[p^\infty],\mathrm{can}\colon \mathcal B'_0\by{\sim} M'_0[p^\infty] )$ descends to an object $(M_0=[L_0\to G_0],\mathcal B, \varepsilon_0\colon \mathcal B_0\by{\sim}M_0[p^\infty] )  )$ in $\Def$. For any $\sigma\in \Gamma$ we then have an isomorphism in $\mathrm{Def}(R',k')$
\[
(\varphi_{\sigma,0},\psi_\sigma)\colon \left(M'_0,\mathcal B',\mathrm{can}\right)\by{\sim} \left(\sigma^*M'_0,\sigma^*\mathcal B',\sigma^*\mathrm{can}\right),
  \]
  where  \[\varphi_{\sigma,0}\colon M'_0\by{\sim}\sigma^* M'_0, \qquad
\psi_\sigma\colon \mathcal B'\by{\sim} \sigma^* \mathcal B',
\]
make the following diagram
\[\xymatrix@R=0.6cm{\mathcal B'_0\ar[rr]^{\psi_{\sigma,0}}\ar[d]_{\mathrm{can}}& & \sigma^* \mathcal B'_0\ar[d]^{\sigma^*\mathrm{can}}\\
 M'_0[p^\infty]\ar[rr]^{\varphi_{\sigma,0}[p^\infty]}&&\sigma^* M'_0[p^\infty]
}\]
commute. By the fully faithfulness of \eqref{st} the isomorphism $(\varphi_{\sigma,0},\psi_\sigma)$ gives a unique isomorphism $\varphi_\sigma\colon  M'\by{\sim}\sigma^* M'$ which lifts $\varphi_{\sigma,0}$ and restricts to $\psi_\sigma$ on BT groups. 
Hence we have defined an action of $\Gamma$ on $M'$, in particular on $L'$ and on $G'$.
Since $L'$ is \'etale over $k'$, it    descends to a lattice $L$ whose special   fiber is $L_0$. On the other hand,  the restriction of the $\Gamma$-action on   weight $\leq -1$  gives a  $G$ in $\M(R)_{\leq -1}$ by Theorem \ref{t.stab}. Finally, since $u'$ is $\Gamma$-equivariant, it descends to a morphism $u\colon L\to G$.  By construction now the $1$-motive $M'$ descends to $M=[u\colon L\to G]$ and the image of $M$ in $\Def$ is  $(M_0,\mathcal B, \varepsilon_0  )$.
\end{proof}

\begin{propose}\label{surj}
The functor \eqref{st} is essentially surjective.
\end{propose}
\begin{proof}
   Let $(M_0=[L_0\to G_0],\mathcal B, \varepsilon_0)$ be an object of $\Def$.
Thanks to Lemma \ref{l.desc} we may assume that $L_0$ is constant and that the maximal subtorus of $G_0$ is split. The conclusion follows then by Lemma \ref{l.lz}.
\end{proof}

With this proposition the proof of Theorem \ref{t.st} is completed.  We should add that some results on deformations of $1$-motives were proved, with different methods, in Madapusi's thesis \cite{MS}, but were not published later.
\medskip

Since any  extension of an \'etale BT group by a toroidal BT group is split over $k$, we deduce the following generalization of Lemma \ref{l.tm}.

\begin{cor} 
Let $T$ be an $R$-torus and $L$ a lattice. For any $1$-motive $M_0=[u_0\colon L_0\to T_0]$ and any BT group $\mathcal B$ which is extension of $L[p^\infty]$ by $T[p^\infty]$ there is a unique $1$-motive $M=[u\colon L\to T]$  which lifts $M_0$ and whose BT group is isomorphic to $\mathcal B$. 
\end{cor}

\subsection{Serre\,-Tate moduli for $1$-motives}\label{moduli}
Let now $k$ be an algebraically closed field of characteristic $p>0$.  Following  \cite[\S~2]{Ka}  one can define the formal moduli of a given $1$-motive $M_0$ over $k$.  Namely let $\widehat{\mathscr{M}}_{M_0}=\widehat{\mathscr{M}}$ be the functor 
\[
\widehat{\mathscr{M}}(R):=\{R\text{-liftings of }M_0\} / \text{iso}
\]
where $R$ is a local artin ring with residue field $k$.  By Theorem~\ref{t.st}  we get a bijection
\[
\widehat{\mathscr{M}}(R)=\{R\text{-liftings of }M_0[p^\infty]\} / \text{iso}\ .
\]

If $M_0=A_0$ is an ordinary abelian variety a construction of Serre\,-Tate gives an isomorphism of functors
\[
\widehat{\mathscr{M}}(-)\simeq \Hom_{\Z_p}(T_pA_0(k)\otimes T_pA_0^*(k),\widehat{\mathbb{G}}_{\rm m}(-))
\]
where $\widehat{\mathbb{G}}_{\rm m}(R)=1+\mathfrak{m}$ are the principal units of $R$ and $A^*_0$ is the dual abelian variety.

This result  can be naturally extended to $1$-motives as follows.

\begin{propose}
Let $k$ be an algebraically closed field of characteristic $p>0$. Let $M_0$ be a $1$-motive over $k$ such that $\gr_{-1}M_0=A_0$ is an ordinary abelian variety. Let $(R,\mathfrak{m})$ be an artinian local ring with residue field $k$.
\begin{itemize}
	\item[(i)] For any  $1$-motive $M$ over $R$ lifting $M_0$ there exist a canonical $\Z_p$-bilinear form
	\[
		q(M/R,-,-)\colon T_pM_0(k)\otimes T_pM_0^*(k)\to \widehat{\mathbb{G}}_{\rm m}(R)
	\]
	inducing an isomorphism of functors
\[
	\widehat{\mathscr{M}}(-)\simeq \Hom_{\Z_p}(T_pM_0(k)\otimes T_pM_0^*(k),\widehat{\mathbb{G}}_{\rm m}(-))
\]
where $T_pM_0(k):=\varprojlim_n M_0[p^n](k)$.
\item[(ii)] Let $\varphi_0\colon M_0\to N_0$ be a morphism of $1$-motives over $k$ such that $\gr_{-1}M_0,\gr_{-1}N_0$ are ordinary abelian varieties. Let $M$ and $N$ be  liftings over $R$  of $M_0$ and $N_0$, respectively. Then $\varphi_0$ lifts to an $R$-morphism $\varphi\colon M\to N$ if, and only if,
\[
	q(M/R,\alpha,\varphi^*_{0}(\beta^*))=q(N/R,\varphi_{0}(\alpha),\beta^*)
\]
for every $\alpha\in T_pM_0(k)$ and $\beta^*\in T_pM_0^*(k)$. Further, is a lifting exists, it is unique.
\end{itemize} 
\end{propose}
\begin{proof}
The proof goes exactly as in the classical case using the following two ingredients (cf. \cite[p. 152]{Ka}):
\begin{itemize}
	\item the Poincar\'e biextension gives a perfect pairing $M[p^n]\times M^*[p^n]\to \mu_{p^n}$ inducing an isomorphism of functors
	\[
		M_0[p^\infty]^\circ \xrightarrow{\sim }\Hom_{\Z_p}(T_pM_0^*(k),\widehat{\mathbb{G}}_m)\ ,
	\]
	since $M_0[p^\infty]^\et=T_pM_0(k)\otimes_{\Z_p}(\Q_p/\Z_p)$.
	 We denote by $E_M\colon M_0[p^\infty]^\circ \times T_pM_0(k)\to\widehat{\mathbb{G}}_m$ the corresponding pairing.

	\item Since $k$ is algebraically closed,  the BT groups $M_0[p^\infty]^\circ$ and $M_0[p^\infty]^\et$ are split. Thus by \cite[Proposition 2.5 p. 180]{M} 
	\[
		\Ext^1(M_0[p^\infty]^\et,M_0[p^\infty]^\circ)\simeq\Hom_{ \Z_p}(T_pM(k),M_0[p^\infty]^\circ(R))
	\]
	and there exists a unique $\phi_M\in \Hom_{ \Z_p}(T_pM(k),M_0[p^\infty]^\circ(R))$ associated to the isomorphism class of the BT group $M_0[p^\infty]$.
\end{itemize}	

Then we can define $q(M/R,\alpha,\beta):= E_M(\phi_M(\alpha),\beta)$ and follow  word by word the proof in \cite{Ka}.
\end{proof}


\vfill

\end{document}